\documentclass[reqno, 11pt]{amsart}
\usepackage[margin = 1.2 in]{geometry}
\usepackage{amsthm,amsmath,amsfonts}
\usepackage{hyperref}
\usepackage{amssymb}
\usepackage{stmaryrd}
\usepackage{eucal, mathrsfs}
\usepackage[libertine]{newtxmath}

\usepackage{mathrsfs}
\usepackage{todonotes}

\usepackage[nameinlink]{cleveref}

\newtheorem{theorem}{Theorem}[section]
\newtheorem{proposition}[theorem]{Proposition}
\newtheorem{lemma}[theorem]{Lemma}
\newtheorem{corollary}[theorem]{Corollary}
\newtheorem{definition}[theorem]{Definition}
\newtheorem{conjecture}[theorem]{Conjecture}
\newtheorem{remark}[theorem]{Remark}
\newtheorem{question}[theorem]{Question}

\renewcommand{\P}{\mathbb{P}}

\renewcommand{\k}{\mathscr{k}}
\newcommand{\Q}{\mathbb{Q}}

\DeclareMathOperator{\Pic}{Pic}
\renewcommand{\O}{\mathcal{O}}
\DeclareMathOperator{\Aut}{Aut}
\newcommand{\HdgX}{\mathsf{Hur}_{d,g}(X)}
\newcommand{\HdtwoX}{\mathsf{Hur}_{d,2}(X)}

\DeclareMathOperator{\rk}{rk}

\newcommand{\Mgbar}{\overline{\mathcal{M}}_{g}}

\DeclareMathOperator{\ch}{ch}
\DeclareMathOperator{\Sym}{Sym}

\DeclareMathOperator{\Z}{\mathbb{Z}}
\DeclareMathOperator{\N}{\mathbb{N}}
\newcommand{\Hurfixedtwo}{\underline{\mathsf{Hur}}_{d,2}(X)}

\newcommand{\St}{\mathsf{Hur}_{d,g}^{\Diamond}(X)}
\DeclareMathOperator{\Proj}{Proj}

\DeclareMathOperator{\univ}{\mathsf{univ}}

\DeclareMathOperator{\Bog}{Bog}

\newcommand{\Hur}[2]{\mathsf{Hur}_{#1,#2}}
\newcommand{\Hurstable}[2]{\mathsf{Hur}_{#1,#2}^{\diamond}}
\newcommand{\Hurfixed}[2]{\underline{\mathsf{Hur}}_{#1,#2}}
\newcommand{\Hurfixedstable}[2]{\underline{\mathsf{Hur}}_{#1,#2}^{\diamond}}

\author{Gabriel Bujokas \& Anand Patel}
\title{Complete families of generic covers of elliptic curves}
\begin{document}

\maketitle

\section{Introduction} 

\subsection{Objectives} The purpose of this paper is first to pose a question in the guise of  \Cref{conjecture:completecurve}, then to explain why it is important in \Cref{theorem:slopeimplication}, and finally to give evidence in support of the conjecture in \Cref{theorem:evidence}. \Cref{conjecture:completecurve} posits the existence of certain complete families of general curves in the Hurwitz stack $\HdgX$ parametrizing degree $d$ branched coverings of a genus $1$ curve $X$.  We provide evidence for the conjecture when $g=2$ or when $d \leq 5$, and $g$ is large. The large-ness assumption can most likely be removed with more care.  Our intention here is to sketch a general conjectural picture, rather than get caught up in interesting technical details.  

While the $d \leq 5$ verification does use the usual unirational description of spaces of covers, the content of the conjecture seems to us quite orthogonal to the existence of unirational parametrizations.  In particular, when $g=2$ and $d$ is large, the Hurwitz stack under question is certainly not unirational, supporting this intuition.  We list some natural next steps at the end of the paper for interested readers wanting to join the fray.

\subsection{Why we care}
Our reason for thinking about  \Cref{conjecture:completecurve}  is in \Cref{theorem:slopeimplication}: we would obtain a uniform lower bound of $5$ for the slope $s_g$ of $\Mgbar$, i.e. if an effective divisor in $\Mgbar$ has divisor class $a \lambda - b \delta$ then the conjecture implies $a/b \geq 5$.  

The use of covers of elliptic curves is by no means new to the conversation around  $s_g$.  Of particular note is Dawei Chen's work in which he studies the slopes of the Teichm\"uller curves parametrizing covers of elliptic curves with only one branch point \cite{chen2011square}.  \footnote{The authors were directly inspired by this work.}  His investigations ultimately gave a lower bound of $O(1/g)$ for $s_g$, as many other attempts have done. We also recommend reading Chen's  beautiful paper \cite{chen2010covers} as well as the informative survey article by Chen, Farkas, and Morrison \cite{chen2012effective}. 

Our basic message is that we should attempt to better understand the stratification of the Hurwitz stack by the isomorphism types of the auxiliary vector bundles on $X$ naturally attached to a branched cover.  The open dense stratum where vector bundles are regular polystable affords, in all tractable cases, many complete, moving curves. If these curves continue to exist outside the currently tractable zone, we would obtain the bound of $5$ for $s_g$.  If they do not persist,  then this would itself be an interesting phenomenon needing explanation.

\subsection{Notation and assumptions}
We will work over an algebraically closed field $\k$ of characteristic zero. All schemes and stacks will be defined over $\k$.    A smooth projective genus $1$ curve $X$ will be fixed throughout.   The projectivization $\P \mathcal{V}$ of a locally free sheaf $\mathcal{V}$ will mean $\Proj \Sym^{\bullet} \mathcal{V}$, the ``post-Grothendieck'' convention.  We fix throughout an elliptic curve $(X,x)$.  All stacks occurring in this paper are Deligne-Mumford stacks. 

\section{The stack of nodal covers of $X$} 
 
\begin{definition}
    \label{definition:nodalcovers} Let $d \geq 1, g \geq 2$ be a pair of integers.  
    \begin{enumerate}       
        \item The stack $\HdgX$  parametrizes degree $d$ finite flat morphisms \[\alpha: C \to X\] where 
        \begin{enumerate}
        \item $C$ is a connected, at-worst-nodal projective curve of arithmetic genus $g$, and
        \item $\alpha$ is {\sl primitive}, i.e. it does not factor through a nontrivial \'etale cover of $X$.
        \end{enumerate}
        \item We denote by $\varphi: \HdgX \to \Mgbar$ the forgetful morphism defined by \[[\alpha: C \to X]  \mapsto [C].\]
        \item On $\Mgbar$ we let $\lambda \in \Pic (\Mgbar)$ denote the first chern class of the Hodge bundle, and we let $\delta \in \Pic (\Mgbar)$ denote the total boundary divisor -- the sum of each irreducible boundary divisor with multiplicity one.  
        \item We abuse notation and use the symbols $\lambda, \delta$ for the respective $\varphi$-pullbacks in $\Pic (\HdgX)$.
    \end{enumerate}
     
\begin{remark}
    When $g =2$, $\HdgX$ is proper, and when $g > 2$ it is not.  This is  because when $3$ or more branch points collide, the possible domain curves include worse-than-nodal possibilities.  And this is not true for collisions of only two branch points.
\end{remark}
    
\end{definition}

\begin{proposition}
    \label{lemma:basics} 
    Let $d,g \geq 2$.
    \begin{enumerate}
        \item $\HdgX$ is irreducible and has dimension $2g-2$.
        \item For fixed $g$, as $d$ ranges over all positive integers, the union \[\bigcup_{d > 0} \varphi(\HdgX) \subset \Mgbar\] is Zariski-dense.
        \item The group scheme $\Aut(X)$ acts on $\HdgX$ via post-composition, and the morphism $\varphi$ is constant on $\Aut(X)$-orbits.
    \end{enumerate}
\end{proposition}

\begin{proof}
    The first statement follows from the well-known fact that a cover $\alpha: C \to X$ is determined by the location of its branch points and discrete monodromy data.  The second statement follows from observing that $\Hur{d}{g}(X)$ contains the collection of branched covers of $X$ branched over a single point.  The domain curves of this collection of covers, as $d \to \infty$ is known to be Zariski dense in $\mathcal{M}_{g}$ -- see, for example, \cite{chen2011square} for the case of $(X,x)$ a square torus.  (The argument works the same for any $j$-invariant.)   The last statement is obvious.
\end{proof}

\begin{corollary}
    \label{corollary:divisor} If $D \subset \Mgbar$ is an effective divisor, then there exist infinitely many positive integers $d$ for which the inclusion \[\varphi^{-1}(D) \subset \HdgX\] is strict.
\end{corollary}

\begin{proof}
    Immediate from \Cref{lemma:basics} part (2).
\end{proof}

\begin{definition}
    \label{definition:regularpolystable} A rank $r$ and  degree $k$ vector bundle $V$ on $X$ is {\bf regular polystable} if $V = V_{1} \oplus \dots \oplus V_{h}$ where $h = \gcd(r,k)$, and where for all $i$:
    \begin{enumerate}
        \item $\rk V_{i} = r/h$ for all $i$,
        \item $\deg V_{i} = d/h$ for all $i$, 
        \item $V_{i}$ is indecomposable for all $i$, i.e. not a direct sum of smaller bundles, and
        \item $\det V_{i} \neq \det V_{j}$ whenever $i \neq j$.
    \end{enumerate} 
\end{definition}

Recall from Atiyah's classification that if $\gcd(r,k) = 1$ then a rank $r$ and degree $k$ indecomposable bundle is uniquely determined by its determinant.  The condition of being regular polystable is an open condition in families of vector bundles.

\begin{definition}
  \begin{enumerate}
     \item Define $\Bog(\mathcal{E}) \subset \HdgX$ (respectively $\Bog(\mathcal{F})$) to be the closed substack parametrizing $\alpha: C \to X$ such that $\mathcal{E}_{\alpha}$ (resp. $\mathcal{F}_{\alpha}$) is not regular polystable. 
     \item Define \[\St \subset \HdgX\] to be the open substack parametrizing covers $\alpha: C \to X$ such that both $\mathcal{E}_{\alpha}$ and $\mathcal{F}_{\alpha}$ are regular polystable vector bundles on $X$. In other words, \[\St = \HdgX \setminus (\Bog(\mathcal{E}) \cup \Bog(\mathcal{F})).\]
     \item Define \[\Hurfixed{d}{g}(X) \subset \HdgX\] to be the closed substack given by the requirement \[\det \mathcal{E}_{\alpha} \simeq \O_{X}((g-1)x).\] 
      \item Define $\Hurfixedstable{d}{g}(X)$ to be the intersection $\St \cap \Hurfixed{d}{g}(X)$.
     \item We will denote by \[\alpha_{\univ}: \mathcal{C} \to \HdgX \times X\] the universal branched covering and we let $\mathcal{E}_{\univ}, \mathcal{F}_{\univ}$ denote the corresponding universal bundles on $\HdgX \times X$. 
     \item We let $\epsilon_{1,1}, \epsilon_{2} \in \Pic(\HdgX)$ (respectively $\varphi_{1,1}, \varphi_{2}$) denote the pushforwards of the classes \[c_{1}^{2}(\mathcal{E}_{\univ}), c_{2}(\mathcal{E}_{\univ})\] (resp. for $\mathcal{F}_{\univ}$) to $\HdgX$.
     \end{enumerate}
\end{definition}

\begin{remark}
    \label{remark:linearlyindependent}
    \begin{enumerate}
    \item The classes $c_{1}(\mathcal{E}_{\univ})$ and $c_{1}(\mathcal{F}_{\univ})$ are proportional, satisfying $$(d-3)c_{1}(\mathcal{E}_{\univ}) = c_{1}(\mathcal{F}_{\univ}),$$ and therefore so are $\epsilon_{1,1}$ and $\varphi_{1,1}$. 
    \item It is possible to prove (using ideas similar to those in \cite[Section 2.2]{deopurkar2015picard}) that the three classes $\epsilon_{1,1}, \epsilon_{2}, \varphi_{2}$ are linearly independent.  
    \item The vector bundles $\mathcal{N}_{i}$ arising in the Casnati-Ekedahl-Schreyer resolution of $\alpha_{\univ}$ produce divisor classes which are already in the span of $\epsilon_{1,1}, \epsilon_{2}, \varphi_{2}$, by the results in \cite{deopurkar2018syzygy}. 
    \end{enumerate}
\end{remark}

The divisor classes $\lambda$ and $\delta$ are contained in the $3$-dimensional vector space $\Q \langle \epsilon_{1,1}, \epsilon_{2}, \varphi_{2} \rangle$. 

\begin{proposition}
    \label{proposition:linearsystem} The classes $\lambda$ and $\delta$ are in the vector space $\Q \langle \epsilon_{1,1}, \epsilon_{2}, \varphi_{2} \rangle$. Specifically, we have: 
    \begin{align*}
        \label{equation:linearsystem}
        \lambda &=  \ch_{2}(\mathcal{E}_{\univ}) \\
        \delta &= 13 \ch_{2}(\mathcal{E}_{\univ}) + \ch_{2}(\mathcal{F}_{\univ}) - \ch_{2}(\Sym^{2}\mathcal{E}_{\univ})
    \end{align*}
    where we have suppressed the pushforward from $\HdgX \times X$ to $\HdgX$.
\end{proposition}

\begin{proof}
    Similar divisor class calculations have been done elsewhere, for example in the proof of \cite[Proposition 2.8]{deopurkar2018syzygy}.  It is a standard use of the Grothendieck-Riemann-Roch theorem -- we leave the details to the reader.
\end{proof}

\begin{definition}
    \label{definition:completecurve} Let $Y$ be a stack.  We say {\bf there exist plenty of complete curves in $Y$} if for all $n \in \N$ there exists a complete curve $B_{i}$ and non-constant morphism $\beta_{i}:B_{i} \to Y$ such that the morphism $\coprod \beta_{n}: \coprod_{n}B_{n} \to Y$  does not factor through a non-trivial closed substack.
\end{definition}

We pose our main question in the form of a conjecture. 

\begin{conjecture}[Complete Curve Conjecture]
    \label{conjecture:completecurve} There are plenty of complete curves in $\Hurfixedstable{d}{g}(X)$ for any $d,g \geq 2$.
\end{conjecture}

\subsection{Divisor class computations}

\begin{proposition}
    \label{proposition:bogomolov} In $\Pic (\HdgX) \otimes \mathbb{Q}$ we have: 
    \begin{align}
        [\Bog(\mathcal{E})] &=  \epsilon_2 -  \frac{d-2}{2d-2} \epsilon_{1,1}\\
        [\Bog(\mathcal{F})] &=  \varphi_{2} - \frac{1}{2}\varphi_{1,1} + \frac{d-3}{d}\epsilon_{1,1}\\
        (5+6/d)\lambda - \delta &= [\Bog(\mathcal{F})] + \frac{(6-d)(d-1)}{d}[\Bog(\mathcal{E})]
    \end{align}
\end{proposition}

\begin{proof}
    Recall that if $\mathcal{V}$ is a vector bundle of rank $r$ on a stack $Z$ then the Bogomolov expression $\Bog(\mathcal{V})$ is given by \[\Bog(\mathcal{V}) = \frac{1}{2r}c_{1}^{2}\mathcal{V} - \ch_{2}\mathcal{V}.\] 
    If $\mathcal{V}$ is a vector bundle on a family of curves, the pushforward of the Bogomolov expression to the base of the family gives the divisor class of the locus where the vector bundle restricts to an unstable bundle, assuming the general fiber is stable (or regular poly-stabe in the case of genus $1$ curves.) 

    And so the proposition follows from this and \Cref{proposition:linearsystem}.
\end{proof}

\begin{corollary}
    \label{theorem:5plus6d}  The relation 
    \[(5+6/d) \lambda - \delta = 0\] holds in the rational Picard group $\Pic(\St) \otimes \mathbb{Q}$.
\end{corollary}

\begin{proof}
    Immediate from \Cref{proposition:bogomolov}.
\end{proof}

\begin{corollary}
    \label{corollary:slope} Suppose $B$ is a complete curve and $\beta: B \to \St$ is a morphism. Then \[(5+6/d) \deg \beta^{*} \lambda - \deg \beta^{*} \delta = 0.\]
\end{corollary}

\begin{proof}
    Immediate from \Cref{proposition:bogomolov}.
\end{proof}

In light of \Cref{corollary:slope}, it makes sense to pose the slightly weaker conjecture: 
\begin{conjecture}[Elliptic Slope Conjecture]
    \label{conjecture:completeslope} There are plenty of complete curves in $\Hurfixed{d}{g}(X)$ which have slope $5+6/d$.
\end{conjecture}
Clearly, \Cref{conjecture:completecurve} implies \Cref{conjecture:completeslope}.

\begin{theorem}
    \label{theorem:slopeimplication} Suppose $D \subset \Mgbar$ is an effective divisor with divisor class $a \lambda - b \delta$.  If \Cref{conjecture:completeslope} is true for $d$ sufficiently large, then \[a/b \geq 5.\]
\end{theorem}

\begin{proof}
    Assume \Cref{conjecture:completeslope}, and let $d$ be any of the infinitely many integers satisfying \Cref{corollary:divisor}.  Let $\beta: B \to \Hurstable{d}{g}(X)$ be a complete curve realizing the truth of \Cref{conjecture:completeslope}.  Then, since $B$ passes through a general point and since $\varphi^{-1}(D) \neq \Hurfixed{d}{g}(X)$, it follows that \[\deg \beta^{*}(D) = a \deg \beta^{*} \lambda - b \deg \beta^{*} \delta \geq 0.\] 

    As $\beta^{*} \deg \lambda \geq 0$ (in fact strictly so) for a non-constant proper family of curves with smooth general member, the theorem now follows from \Cref{corollary:slope} by letting $d \to \infty$.
\end{proof}

\section{Evidence for the conjectures} It is time to explain our evidence in support of the conjectures. 

\begin{theorem}
    \label{theorem:evidence} \Cref{conjecture:completecurve} holds when $d \leq 5$ and $g \gg 0$, and \Cref{conjecture:completeslope} holds when $g=2$. 
\end{theorem}

\subsection{The genus 2 situation}

If $g=2$, then $\HdtwoX$ is a complete surface and $\Hurfixedtwo$ is a complete curve.

\begin{theorem}
    \label{theorem:genus2} The slope of the curve $\Hurfixedtwo$ is \[\frac{\Hurfixedtwo \cdot \delta}{\Hurfixedtwo \cdot \lambda} = 5+6/d.\]
\end{theorem}

\begin{remark}
    This should be compared to a result of Chen in \cite{chen2011square} in which he establishes that the $d \to \infty$ limit of slopes of Teichm\"uller curves is $5$. 
\end{remark}

For simplicity, let us write $s$ for the slope $\frac{\Hurfixedtwo \cdot \delta}{\Hurfixedtwo \cdot \lambda}$. For $i=0,1$ let $\Delta_{i}(d) = \Hurfixedtwo \cdot \delta_{i}$, where $\delta_{0},\delta_{i}$ are the two boundary divisors in $\overline{\mathcal{M}}_{2}$.  From the well-known relation \[10 \lambda = \delta_{0} + 2 \delta_{1},\] it follows that \[s = \frac{10(\Delta_{0}+\Delta_{1})}{\Delta_{0}+2\Delta_{1}} = 5 + \frac{5\Delta_{0}}{\Delta_{0}+2\Delta_{1}}.\]

Therefore, the claim in \Cref{theorem:genus2} is equivalent to showing 
\begin{align}
\label{equation:desired}
    (5d-6)\Delta_{0}(d) = 12\Delta_{1}(d).
\end{align}

\begin{definition}
    Fix a point $x \in X$, and an integer $i \leq \lfloor d/2 \rfloor$.  
    \begin{enumerate}
    
    \item Define the set $A_{d}(x)$ to be the set of isomorphism classes of (primitive) covers $\alpha: Y \to X$ where $Y$ is an irreducible curve with a single node $y$ satisfying $\alpha(y)=x$.
    
        \item Define the set $B(i,x)$ to be the set of isomorphism classes of (primitive) covers $\alpha: X_{1} \cup X_{2} \to X$ such that $X_{1},X_{2}$ are genus $1$ curves meeting at a single node $y$ satisfying $\alpha(y)=x$, and such that $\alpha$ restricts to a degree $i$ isogeny on one of the curves $X_{1},X_{2}$.  Let $B_{d}(x) = \cup_{i}B(i,x)$

        \item Let $C_{d}(x)$ denote the set of degree $d$ pointed isogenies $\beta: (X',x') \to (X,x)$ where $X'$ is a smooth genus $1$ curve.
        
    \end{enumerate}
\end{definition}

\begin{lemma}
    \label{lemma:Deltas} We have 
    \begin{align}
        \label{equation:Deltas} \Delta_{0}(d) &= 4 \cdot |A_{d}(x)| \\
        \Delta_{1}(d) &= 4 \cdot |B_{d}(x)| 
    \end{align}
\end{lemma}

For integers $k,n$, let $\sigma_{k}(n) = \sum_{d \mid n}d^{k}.$

\begin{lemma}
    \label{lemma:isogenies} \[|C_{d}(x)| = \sigma_{1}(d).\]
\end{lemma}

\begin{proof}
    $|C_{d}(x)|$ can be identified with the number of lattices $\Lambda \subset \Z^{2}$ which have co-volume $d$. We may canonically choose a basis of $\Lambda$ for which the inclusion $\Lambda \subset \Z^{2}$ is represented by a matrix \[\begin{pmatrix}
        a & c \\
        0 & b
    \end{pmatrix}\] where $a, b > 0$, $ab=d$, and $0 \leq c < a$.  For each factorization $ab=d$ there are $a$ possibilities for $c$, and hence the number of such sublattices is $\sum_{a \mid d}a = \sigma_{1}(d)$.
\end{proof}

\begin{proposition}
    \label{proposition:recursion} The integers $|B_{d}(x)|$, $d=1,2,\dots$ satisfy the recursive formula: 
    \begin{align}
        \label{equation:recursion} \sum_{d'|d}|B_{d'}(x)|\sigma_{1}(d/d') = \frac{1}{2}\left[ \sigma_{1}(d/2) + \sum_{h=1}^{d-1}\sigma_{1}(h)\sigma_{1}(d-h). \right]
    \end{align}
    where we set $\sigma_{1}(d/2) = 0$ when $d$ is odd.
\end{proposition}

\begin{proof}
    Let $x \in X$ denote a marked point. The proof proceeds by counting the elements of a set $D$ in two different ways.  The set $D$ is the collection of (isomorphism classes of) all degree $d$ branched covers (not necessarily primitive) $\alpha: X_{1} \cup X_{2} \to X$ where 
    \begin{enumerate}
        \item $X_{1},X_{2}$ are smooth genus $1$ curves meeting at a single node $y$,
        \item $\alpha(y) = x$.
    \end{enumerate}

    On the one hand, if the degree of the isogeny $\alpha|_{X_{1}}:(X_{1},y) \to (X,x)$ is $h$, then there are $\sigma_{1}(h)$ choices for this isogeny by \Cref{lemma:isogenies}. Similarly, there are $\sigma_{1}(d-h)$ choices for the isogeny $\alpha|_{X_{2}}:(X_{2},y) \to (X,x)$. If $h \neq d-h$ then, accounting for labeling, it follows that there are $1/2(\sigma_{1}(h)\sigma_{1}(d-h))$ choices for a combined branched cover $\alpha:X' \to X$ where $X'$ has two components, one of which maps with degree $h$.  

    If $h = d-h$ then, after accounting for labeling, there are ${\sigma_{1}(h) \choose 2}$ choices of branched covers $\alpha:X_{1} \cup X_{2} \to X$ where the isogenies $\alpha|_{X_{1}}$ and $\alpha|_{X_{2}}$ are distinct.  If we then add in the $\sigma_{1}(h)$-many branched covers where $\alpha|_{X_{1}} = \alpha|_{X_{2}}$ we obtain a contribution of $\sigma_{1}(h) + {\sigma_{1}(h) \choose 2} = \frac{1}{2}\sigma_{1}(h) + \frac{1}{2}\sigma_{1}(h)^{2}$.  Adding all contributions per each $h \in \{1, \dots, d-1\}$ yields the right hand side of \eqref{equation:recursion}.  

    On the other hand, every element $\alpha: X_{1} \cup X_{2} \to X$ in $D$ factors canonically through an isogeny $(X_{1} \cup X_{2},y) \to (Z,z) \to (X,x)$, $Z$ a smooth genus $1$ curve, in such a way that $X_{1} \cup X_{2} \to Z$ is primitive.  By counting with respect to the degree $d'$ of the factor $X_{1}\cup X_{2} \to Z$ we obtain the left side of \eqref{equation:recursion}.  This concludes the proof.
\end{proof}

\begin{definition}
    \label{definition:multiplicative} A function $f:\Z_{>0} \to \Z$ is \emph{multiplicative}  if $f(ab) = f(a)f(b)$ whenever $\gcd(a,b)=1$.
\end{definition}

Of course, a multiplicative function is uniquely determined by its values at pure prime powers $p^{n}$.

\begin{definition}
\label{definition:markedcoke}
    Let $M_{d}$ denote the set of ordered pairs $(\Lambda,\eta)$ satisfying the following: 
    \begin{enumerate}
        \item $\Lambda$ is a sub-lattice of $\Z^{2}$ of index $d$, and
        \item $\Z^{2}/\Lambda$ is a non-trivial cyclic group with $\eta \in \Z^{2}/\Lambda$ a generator.
    \end{enumerate}
\end{definition} 

\begin{lemma}
    \label{lemma:countAd} 
    \begin{enumerate}
    \item If $d \geq 3$ then \[|A_{d}(x)| = \frac{1}{2}|M_{d}|.\]
    \item If $d=2$ then $|A_{2}(x)| = |M_{2}(x)|$.  
    \item If $d=1$ then $|A_{1}(x)| = 0 = |M_{1}(x)|$.
    \end{enumerate}
\end{lemma}

\begin{proof}
    The case $d=1$ is immediate, as both sets $A_{1}(x),M_{1}$ are clearly empty. So we assume $d \geq 2$. We will produce a function $\Phi: M_{d} \to A_{d}(x)$ which will explain the remaining statements. First, we identify $\pi_{1}(X,x)$ with $\Z^{2}$.  Next, pick any element $(\Lambda,\eta) \in M_{d}$.
    
    By the fundamental Galois correspondence in the theory of covering spaces, the subgroup $\Lambda \subset \Z^{2}$ defines a degree $d$ connected, Galois  covering space $\epsilon: Z \to X$ with group of deck transformations canonically isomorphic to $\Z^{2}/\Lambda$.  $Z$ inherits from $X$ a complex structure, and therefore $Z$ is a genus 1 curve with $\epsilon$ a degree $d$ finite \'etale map.  Pick a point $z \in \epsilon^{-1}(x)$. The chosen element $\eta \in \Z^{2}/\Lambda$ defines a second point $z' = \eta(z)$ on $Z$.  Let $Y_{(\Lambda,\eta)}$ denote the nodal curve obtained by identifying the points $z$ and $z'$ on the curve $Z$ -- $Y_{(\Lambda,\eta)}$ is independent of the choice of $z$ because the cover $Z \to X$ is Galois. The map $\epsilon$ descends to a finite, flat degree $d$ map \[\alpha:Y_{(\Lambda,\eta)} \to X\] where the node of $Y_{(\Lambda,\eta)}$ maps to $x$.  Since $\eta$ generates $\Z^{2}/\Lambda$, it follows that $\alpha: Y_{(\Lambda,\eta)} \to X$ does not factor through a non-trivial \'etale cover $X' \to X$, again by the Galois correspondence. 

    Thus, $\alpha: Y_{(\Lambda,\eta)} \to X$ is an element of $A_{d}(x)$. The function $\Phi: M_{d} \to A_{d}(x)$ is defined by the rule \[(\Lambda,\eta) \mapsto [\alpha:Y_{(\Lambda,\eta)} \to X].\] 
    
    It suffices to show that $\Phi$ is surjective and  $2:1$ when $d \geq 3$, and $1:1$ when $d=2$. To that end, let $\alpha: Y\to X$ be an arbitrary element of $A_{d}(x)$, let $y \in Y$ denote the node of $Y$, and let $Z \to Y$ denote the normalization of $Y$.  $Z$ is a genus $1$ curve, and there are two points $z, z' \in Z$ lying above the node $y \in Y$. Again, since $\alpha: Y \to X$ is primitive, it follows that the induced (Galois) cover $\epsilon: Z \to X$ has a cyclic Galois group $G$. Hence, there is unique element $\eta_{1} \in G$ which satisfies $\eta(z) = z'$. Setting $\Lambda_{1}$ to be the image of the pushforward map $\epsilon_{*}:\pi_{1}(Z,z) \to \pi_{1}(X,x)$, we see that $\alpha = \Phi(\Lambda_{1}, \eta_{1})$.  

    Now suppose $\alpha = \Phi(\Lambda_{2}, \eta_{2})$. Then $\Lambda_{2} = \Lambda_{1}$ because both are simply the image of $\pi_{1}(Z)$ under push-forward $\epsilon_{*}$ (the normalization is intrinsic to $Y$).  Now suppose $w,w' \in Z$ are two points in $\epsilon^{-1}(x)$ such that $\eta_{2}(w)=w'$. Since $\alpha = \Phi(\Lambda_{2},\eta_{2})$ it follows by construction of $\Phi$ that $\alpha:Y \to X$ can be obtained by identifying the points $w$ and $w'$, creating the node $y$. By uniqueness of normalization, it follows that there is an automorphism  $g: Z \to Z$ (\emph{over $X$}) which sends the set $\{z,z'\}$ to the set $\{w,w'\}$.  Since the cover $\epsilon: Z \to X$ is Galois, it follows that $\eta_{2}$ either sends $z$ to $z'$ or $z'$ to $z$.  In the former case $\eta_{2}=\eta_{1}$ and in the latter $\eta_{2} = \eta_{1}^{-1}$.  The proof is finished as we have shown that $\Phi$ is surjective and $2:1$ when $d \geq 3$ and $1:1$ when $d=2$.
\end{proof}

\begin{proposition}
    \label{proposition:Deltazeromultiplicative} Let $\delta_{i,j}$ be the Kronecker delta function. Then the function \[F(d) := 2|A_{d}(x)| + \delta_{d,1}-3\delta_{d,2}\] is multiplicative, and if $p$ is a prime number then \[F(p^{n}) = (p^{2}-1)p^{2n-2}.\]
\end{proposition}

\begin{proof}
    The proposition will fall out of carefully counting the elements of $M_{d}$ and then using \Cref{lemma:countAd} -- the Kronecker deltas will correspond to the two edge cases in \Cref{lemma:countAd}. 
    
    First, the set of lattices $\Lambda \subset \Z^{2}$ of index $d$ is in natural bijection with the set of integer matrices \[\begin{pmatrix}
        a & c \\
        0 & b
    \end{pmatrix}\] where $a,b > 0$, $0 \leq c < a$, and $ab=d$. In terms of this bijection, the condition that $\Z^{2}/\Lambda$ is cyclic is expressed by the condition that $\gcd(a,b,c)=1$. 

    For each choice of factorization $ab=d$, the condition $\gcd(a,b,c)=1$ can be replaced by the condition: \emph{any prime $p$ dividing both $a$ and $b$ does not divide $c$.} Thus, for a fixed $a$ dividing $d$, the number of possible choices for $c$ is \[ a \cdot \prod_{p \mid a \,\, \text{and} \,\, p \mid b}\left(1 - \frac{1}{p} \right).\] Summing over all factorizations we get 
    \begin{align*}
    \sum_{ab=d}\left[ \prod_{p \mid a \,\, \text{and} \,\, p \mid b}\left(1 - \frac{1}{p} \right) \right] &= \prod_{p^{n}\mid \mid d}\left[1 + \left( \sum_{i=1}^{n-1}p^{i}\left(1-\frac{1}{p}\right) \right) + p^{n} \right]\\
    &= \prod_{p^{n} \mid \mid d}(p^{n}+p^{n-1}),
    \end{align*}
    where ``$p^{n} \mid \mid d$'' means ``$p^{n}$ is the largest power of $p$ dividing $d$."

    Now, accounting for the $\eta$ part of the data of an element $(\Lambda,\eta) \in M_{d}$, we conclude that for $d \geq 2$, \[|M_{d}| = \phi(d)\prod_{p^{n} \mid \mid d}(p^{n}+p^{n-1}) = \prod_{p^{n} \mid \mid d}p^{2n-2}(p^{2}-1)\] where $\phi$ is the Euler phi function, and hence by \Cref{lemma:countAd} we conclude that when $d \geq 3$, 
    \begin{align}
    \label{equation:almost}
    2|A_{d}(x)| = \prod_{p^{n} \mid \mid d}p^{2n-2}(p^{2}-1).
    \end{align}

    When $d=1$ \eqref{equation:almost} is almost right: The left side is $0$ while the right side is $1$ because it is a product over an empty set.  When $d=2$ \eqref{equation:almost} is also almost right: The left side should be replaced by $|A_{2}(x)|$ because of the $d=2$ exception in \Cref{lemma:countAd}.  When we adjust $2 | A_{d}(x)|$ by the Kronecker deltas in the statement of the proposition, we find that the function \[F(d) := 2|A_{d}(x)|+\delta_{d,1}-3\delta_{d,2}\] has output $\prod_{p^{n} \mid \mid d}p^{2n-2}(p^{2}-1)$, which is plainly multiplicative in $d$. The proposition follows.
\end{proof}

\begin{proof}[Proof of \Cref{theorem:genus2}]
Recall that it suffices to prove \eqref{equation:desired}, and hence by \Cref{lemma:Deltas} it suffices to prove 
\begin{align}
    \label{equation:AB}
    (5d-6)|A_{d}(x)| = 12|B_{d}(x)|.
\end{align}
In order to prove \eqref{equation:AB} we will show that the numbers $\frac{5d-6}{12}|A_{d}(x)|$ satisfy the recursion \eqref{equation:recursion} in \Cref{proposition:recursion} after replacing $|B_{d'}(x)|$ with $\frac{5d'-6}{12}|A_{d'}(x)|$ on the left side.  Denote by $\iota: \Z_{>0} \to \Z$ the inclusion map, which is clearly multiplicative.

\emph{On one hand}, in light of \Cref{proposition:Deltazeromultiplicative}, we get 
\begin{align}
    \label{equation:rewriteA} \sum_{d' \mid d}\frac{(5d'-6)|A_{d'}(x)|}{12}\sigma_{1}(d/d') &=\\
    \frac{1}{24}\sum_{d' \mid d} (5d'-6)F(d')\sigma_{1}(d/d')  - \frac{1}{24}(5 \times 1 - 6)\sigma_{1}(d) + \frac{1}{24}\times 3 \times (5 \times 2 - 6)\sigma_{1}(d/2)\\
    = \frac{5}{24}\left[ (\iota \cdot F) * \sigma_{1} \right](d) - \frac{1}{4}\left[ F * \sigma_{1} \right](d) + \frac{1}{24}\sigma_{1}(d) + \frac{1}{2}\sigma_{1}(d/2).
\end{align}

Here and in what follows, $f \cdot g$ is the ordinary multiplication of two numerical functions and $f*g$ denotes the Dirichlet convolution of two numerical functions, given by $(f * g)(d) = \sum_{d' \mid d}f(d')g(d/d')$.  If $f$ and $g$ are both  multiplicative then so are $f\cdot g$ and $f*g$. 

\emph{On the other hand}, we can simplify the right side of \eqref{equation:recursion} using the Ramanujan identity 
\[\sum_{h=1}^{d-1}\sigma_{1}(h)\sigma_{1}(d-h) = \left(\frac{1}{12}-\frac{d}{2}\right)\sigma_{1}(d)+ \frac{5}{12}\sigma_{3}(d).\]  Thus it remains to prove the equality: 
\begin{align*}
    \frac{1}{2}\left(\sigma_{1}(d/2) + \left(\frac{1}{12}-\frac{d}{2}\right)\sigma_{1}(d)+ \frac{5}{12}\sigma_{3}(d)\right) = \\
    \frac{5}{24}\left[(\iota \cdot F) * \sigma_{1} \right](d) - \frac{1}{4}\left[ F * \sigma_{1} \right](d) + \frac{1}{24}\sigma_{1}(d) + \frac{1}{2}\sigma_{1}(d/2).
\end{align*}

This further simplifies, and it then suffices to show the equality: 
\begin{align}
    \label{equation:suffices}
    \frac{5}{12}\sigma_{3}(d) - \frac{d}{2}\sigma_{1}(d)  = \\
    \frac{5}{12}\left[(\iota \cdot F) * \sigma_{1} \right](d) - \frac{1}{2}\left[ F * \sigma_{1} \right](d). 
\end{align}

\Cref{equation:suffices} then follows from the following two simple identities, both proven using the fact that the Dirichlet convolution of multiplicative functions is multiplicative, and then by evaluating each expression at prime powers: 
\begin{enumerate}
    \item $\left[(\iota \cdot F) * \sigma_{1} \right](d) = \sigma_{3}(d)$,
    \item $\left[ F * \sigma_{1} \right](d) = d \cdot \sigma_{1}(d)$
\end{enumerate}
This finishes the verification that the numbers $\frac{5d-6}{12}|A_{d}(x)|$ satisfy the recursion \eqref{equation:recursion} for $|B_{d}(x)|$, and therefore it follows that \eqref{equation:desired} is satisfied, finishing the proof of \Cref{theorem:genus2}.
\end{proof}

\subsection{Verification of the conjecture in low degrees} Our next goal is to verify \Cref{conjecture:completecurve} for $d \leq 5$. The main tool we will use is the well-known parametrization of moduli spaces of covers of low degrees. We reference the fundamental papers of Casnati and Ekedahl here: \cite{casnatiekedahlI}, \cite{casnati1996covers}.

\begin{theorem}
    \label{theorem:trigonal} 
\begin{enumerate}
    \item Choose a general polystable bundle $\mathcal{E}$ of rank $2$ and degree $g-1$ on $X$. In the surface scroll $\P \mathcal{E}$, choose a general line $B$ in the projective space $\P (H^{0}(X, \Sym^{3} \mathcal{E} \otimes (\det \mathcal{E})^{\vee}))$.  Then $B$ defines a complete, moving curve in $\Hurfixedstable{3}{g}(X)$.
     \item Choose a general polystable bundle $E$ of rank $3$ and degree $g-1$ on $(X,x)$ with determinant $(g-1)x$, and then choose a general polystable rank $2$ vector bundle $\mathcal{F}$ on $X$ with $\det \mathcal{F} = \det E$. A general pencil $B$ in the projective space $\P (H^{0}(X, \Sym^{2}\mathcal{E} \otimes \mathcal{F}^{\vee}))$ defines a complete, moving curve in $\Hurfixedstable{4}{g}(X)$.
     \item Choose a general polystable bundle $\mathcal{E}$ of rank $4$ and degree $g-1$ on $X$ with determinant $(g-1)x$, and then choose a general polystable bundle $\mathcal{F}$ of rank $5$ satisfying $\det \mathcal{F} = (\det \mathcal{E})^{\otimes 2}$. A general pencil $B$ in  $\P (H^{0}(X, \wedge^{2} \mathcal{F} \otimes \mathcal{E} \otimes (\det \mathcal{E})^{\vee}))$ defines a complete, moving curve in $\Hurfixedstable{5}{g}(X)$.
\end{enumerate}
\end{theorem}

\begin{proof}
    \begin{enumerate}
        \item The pencil $B$ can be interpreted as a linear map $$\lambda: \P^{1} \to \P (H^{0}(X, \Sym^{3} \mathcal{E} \otimes (\det \mathcal{E})^{\vee})).$$ We may interpret $\lambda$ as a strict injection \[\O_{\P^{1}}(-1) \hookrightarrow \Sym^{3} \mathcal{E} \otimes (\det \mathcal{E})^{\vee}.\] But this, in turn, can be interpreted as a single section  \[t \in H^{0}\left( X \times \P^{1}, \Sym^{3} (\mathcal{E'}) \otimes (\det \mathcal{E'})^{\vee} \right)\] where $\mathcal{E}'$ is short for $\mathcal{E} \boxtimes \O_{\P^{1}}(1)$.  Conversely, a general section $t$ can be interpreted as coming from a pencil $B$, and so we obtain a general section $t$ in this way.  
        
        The large $g$ assumption ensures that $\Sym^{3} (\mathcal{E'}) \otimes (\det \mathcal{E'})^{\vee}$ is globally generated, in fact even very ample.  It follows from Casnati and Ekedahl's Bertini-type theorems in \cite{casnatiekedahlI,casnati1996covers} that $t$ defines a smooth surface $S \subset \P \mathcal{E}'$ which is a triple cover of $X \times B$ with Tschirnhausen bundle $\mathcal{E}'$. The numerics of the initial bundle $\mathcal{E}$ ensures that the natural morphism $S \to \P^{1}$ is a family of arithmetic genus $g$ curves, with smooth general fiber.   
        
        Furthermore, again as a consequence of the large $g$ condition, the set of elements of $\P (H^{0}(X, \mathcal{E}\otimes (\det \mathcal{E})^{\vee}))$ which define a worse-than-nodal curve in $\P \mathcal{E}$ has codimension at least $2$.  Therefore, a general pencil avoids this locus and so we obtain in this way a map $B \to \Hurstable{3}{g}(X)$ which is clearly moving. 
        
        \item We proceed in much the same way as in the previous case. A general pencil $B \subset \P (\Sym^{2}\mathcal{E} \otimes \mathcal{F}^{\vee})$ can instead be interpreted as a single general global section $t$ of $\left( \Sym^{2} \mathcal{E} \otimes \mathcal{F}^{\vee} \right) \boxtimes \O_{\P^{1}}(1)$ on $X \times \P^{1}$.  Let $\mathcal{E}'$ and $\mathcal{F}'$ denote the bundles $\mathcal{E} \boxtimes \O_{\P^{1}}(2)$ and $\mathcal{F} \boxtimes \O_{\P^{1}}(3)$, respectively, on $X \times \P^{1}$.  Then $t$ is a general global section of $\Sym^{2} \mathcal{E}' \otimes (\mathcal{F}')^{\vee}$, and $\det \mathcal{E}' = \det \mathcal{F}'$. 
        
        The large $g$ condition ensures $\Sym^{2} \mathcal{E}' \otimes (\mathcal{F}')^{\vee}$ is globally generated. By Casnati and Ekedahl's Bertini-type theorems \cite{casnatiekedahlI,casnati1996covers}, $t$ defines a smooth surface $S \subset \P \mathcal{E}'$, which is a degree $4$ branched cover of $X \times B$ with Tschirnhausen bundle $\mathcal{E}'$. The morphism $S \to B$ gives a family of arithmetic genus $g$ curves with smooth general fiber.

        The positivity of $\Sym^{2}\mathcal{E} \otimes \mathcal{F}^{\vee}$ implies that the locus of elements in its projectivization defining worse-than-nodal curves has codimension at least $2$. So $B$ can be made to avoid this locus, and therefore $B$ defines a moving curve in $\Hurfixedstable{4}{g}(X)$.

        \item Proceed in parallel to the previous two cases, with the relevant bundles on $X \times \P^{1}$ being $\mathcal{E}' = \mathcal{E} \boxtimes \O_{\P^{1}}(5)$ and $\mathcal{F}' = \mathcal{F} \boxtimes \O_{\P^{1}}(8)$. We omit the repetitive details.
    \end{enumerate}
\end{proof}

\section{Final remarks and possible next steps} 

In \Cref{theorem:genus2} we computed that the slope of the Hurwitz curve $\Hurfixed{d}{2}(X)$ is $5+6/d$, and the proof used lattice counting arguments.  On the other hand, the unique (up to scaling) linear combination of the divisors $[\Bog(\mathcal{E})]$ and $[\Bog(\mathcal{F})]$ lying in the $\lambda$-$\delta$ plane has slope $5+6/d$, according to \Cref{proposition:bogomolov}.  {\sl Why are these two numbers the same?}

We understand half of the answer to this riddle.  When $g=2$, the Tschirnhausen bundle $\mathcal{E}$ has degree $1$, and so poly-stability just means that $\mathcal{E}$ is indecomposable.  In \Cref{theorem:genus2stableE} below we will in fact show that $\Bog(\mathcal{E})$ is empty when $g = 2$.  Given \Cref{proposition:bogomolov}, if we knew that $\mathcal{F}_{\alpha}$ is regular polystable for a general $\alpha$ then by \Cref{theorem:genus2} we would conclude that $\Bog(\mathcal{F_{\alpha}})$ must also be empty, and therefore that \Cref{conjecture:completecurve} holds for $g=2$. But we do not know this. In general, it would be great to know the answer to the following question: 

\begin{question}
    \label{question:stableF}
    Does a general (primitive) branched cover $\alpha: C \to D$ between two curves of genus $\geq 1$ have a stable Casnati-Ekedahl bundle $\mathcal{F}$? What about the other terms of the Casnati-Ekedahl resolution? 
\end{question}

\begin{remark}
    The stability of the general Tschirnhausen bundle was demonstrated beautifully in \cite{coskun2024stability}. In another direction, \cite{deopurkar2022vector} proves that the up to twist, all conceivable rank $d-1$ vector bundles eventually arise as Tschirnhausen bundles, as $g$ increases. 
\end{remark}

Throughout this section, we let $\alpha: C \to X$ be an element of $\Hurfixed{d}{2}(X)$ -- in particular recall that $\alpha$ is primitive. 

\begin{lemma}
    \label{lemma:degree0twist}
    If $\mathcal{L}$ is any degree $0$ line bundle on $X$ then $$h^{0}(X, \mathcal{L} \otimes \mathcal{E}_{\alpha}^{\vee}) = 1.$$
\end{lemma}

\begin{proof}
    By definition, \[\alpha_{*} \O_{C} \simeq \O_{X} \oplus \mathcal{E}_{\alpha}^{\vee}.\]  Therefore, \[h^{0}(X,\mathcal{L}) + h^{0}(X, \mathcal{L} \otimes \mathcal{E}_{\alpha}^{\vee}) = h^{0}(X, (\alpha_{*} \O_{C})\otimes \mathcal{L}) = h^{0}(C, \alpha^{*} \mathcal{L}),\]  using the push-pull formula.  Now, $\alpha^{*} \mathcal{L}$ is a degree $0$ line bundle on $C$, and so $h^{0}(C, \alpha^{*} \mathcal{L}) = 1$ if and only if $\alpha^{*} \mathcal{L} \simeq \O_{C}$, and otherwise $h^{0}(C, \alpha^{*} \mathcal{L}) = 0$.  

    The pushforward of the divisor $c_{1} \alpha^{*} \mathcal{L}$ is $d \cdot c_{1} \mathcal{L}$, and therefore if $\alpha^{*} \mathcal{L} \simeq \O_{C}$ then $\mathcal{L}$ is $d$-torsion in $\Pic^{0}(X)$. 
    
    Suppose $e \mid d$ is the smallest positive integer such that $\mathcal{L}$ is $e$-torsion. Then $\mathcal{L}$ defines a a degree $e$ finite et\'ale cover $\beta: \tilde{X} \to X$ with the property that: if $\gamma: Z \to X$ is any morphism such that $\gamma^{*} \mathcal{L} \simeq \O_{Z}$ then $\gamma$ factors through $\beta$. In particular, $\alpha$ factors through $\beta$, and therefore by primitivity $e=1$.  To summarize, $\alpha^{*} \mathcal{L} \simeq \O_{C}$ if and only if $\mathcal{L} \simeq \O_{X}$.  The lemma follows.
\end{proof}

\begin{theorem}
\label{theorem:genus2stableE}
    The Tschirnhausen bundle $\mathcal{E_{\alpha}}$ is indecomposable for \underline{any} element $[\alpha: C \to X] \in \Hurfixed{d}{2}(X)$.
\end{theorem}

\begin{proof}
    Write $\mathcal{E}_{\alpha}^{\vee}$ as a direct sum of indecomposable bundles \[\mathcal{E}_{\alpha}^{\vee} = \bigoplus_{i=1}^{r} \mathcal{A}_{i}.\]  Since $h^{0}(X,\mathcal{E}_{\alpha}^{\vee}) = 0$ ($C$ is connected), none of the summands $\mathcal{A}_{i}$ can have positive degree.  

    Suppose one particular $\mathcal{A}_{j}$ has degree $0$.  By Atiyah's classification, $\mathcal{A}_{j}$ is the tensor product of the bundle $\mathcal{F}$ by some line bundle $\mathcal{M}$ of degree $0$.  Here $\mathcal{F}$ is the unique indecomposable vector bundle of rank $\rk \mathcal{A}_{j}$ and degree $0$ which has a non-trivial global section.  If we set $\mathcal{L} := \mathcal{M}^{-1}$, then $\mathcal{A}_{j} \otimes \mathcal{L} = \mathcal{F}$, and so  $h^{0}(X, \mathcal{A}_{j} \otimes \mathcal{L}) = 1$.  And so \Cref{lemma:degree0twist} is violated.  

    Therefore, all summands $\mathcal{A}_{i}$ have strictly negative degree, and these degrees sum to $\deg \mathcal{E}_{\alpha}^{\vee} = -1$.  Hence there is only one summand, and the theorem follows.
\end{proof}

Like we said earlier, we do not know whether $\mathcal{F}_{\alpha}$ is regular polystable for a general $[\alpha: C \to X] \in \Hurfixed{d}{2}(X)$ -- this is a natural thing to look into.  An affirmative answer to \Cref{question:stableF} means that the ability to construct complete families of branched covers of $X$ lying entirely inside $\Hurfixedstable{d}{g}(X)$ has little to do with the existence of nice parametrizations -- the curves $\Hurfixed{d}{2}(X)$ are themselves examples in genus $2$.  This in turn gives us more confidence in \Cref{conjecture:completecurve}.

We leave the reader with some ideas for next steps:

\begin{enumerate}
    \item  Is the general Casnati-Ekedahl bundle $\mathcal{F}$ stable (or regular polystable in the case of a genus $1$ target)? Start with the genus $2$ case.  We'd like to point out the main theorem of \cite{bujokas2021invariants} and \cite[Proposition 2.4]{bujokas2021invariants}.  One must be cautious, in light of \cite[Example 2.3]{bujokas2021invariants}!
    \item  Is the weaker \Cref{conjecture:completeslope} true for the Hurwitz space parametrizing degree $6$ genus $3$ covers, $\Hurfixed{6}{3}(X)$? This is the frontier case, in our eyes.
    \item Can we leverage the success in genus 2 to serve higher genera? Perhaps by constructing complete families in the boundary, and then attempting to deform?
    \item Relaxing the requirements of the complete curve conjecture, can one show that there are plenty of complete curves avoiding $\Bog(\mathcal{E})$ or $\Bog(\mathcal{F})$ individually?
    \item Relaxing the requirements of the complete curve conjecture in another direction: Can we exhibit a single complete curve in every $\Hurfixedstable{d}{g}(X)$, moving or otherwise?
    \item Prove \Cref{theorem:trigonal} after dropping the convenient assumption that $g$ is sufficiently large.
\end{enumerate}

\bibliographystyle{alpha}
 \bibliography{references}

\end{document}